\documentclass[12pt]{article}
\usepackage{amsmath}

\usepackage{amssymb}
\usepackage{graphicx}

\newtheorem{theorem}{Theorem}
\newtheorem{corollary}{Corollary}

\newtheorem{lemma}{Lemma}

\newenvironment{proof}[1][Proof]{\begin{trivlist}
\item[\hskip \labelsep {\bfseries #1}]}{\end{trivlist}}

\begin{document}

\begin{center}\Large
\textbf{New identities for the polarized partitions and partitions with $d$-distant parts}
\bigskip\large

Ivica Martinjak, Dragutin Svrtan

University of Zagreb

Zagreb, Croatia
\end{center}


\begin{abstract} 
In this paper we present a new class of integer partition identities. The number of partitions with $d$-distant parts can be represented as a sum of the number of partitions with 1-distant parts whose even parts are greater than twice the number of odd parts. We also provide a direct bijection between these classes of partitions.
\end{abstract}



\section{Preliminaries and introduction }

A {\it partition} is the sequence $\lambda $=$(\lambda_1, \lambda_2,...,\lambda_l)$,  where $\lambda_i$ is a natural number $\forall i$, assuming $\lambda_1 \ge \lambda_2 \ge ... \ge \lambda_l$. Having $\sum_{i=1}^l \lambda_i$=$|\lambda|$=$n$ we say that $\lambda$ is a partition of $n$, denoted $\lambda \vdash n$. The numbers $\lambda_i$ are the {\it parts} of a partition, the number of parts $l(\lambda)$=$l$ is the {\it length} of a partition. The set of all partitions of $n$ is denoted by ${\cal P}_n$, with $p(n)=|{\cal P}_n|$ and the set of all partitions by ${\cal P}$. The number of partitions $\lambda \vdash n$ with exactly $k$ summands is denoted $p_k(n)$. The number of partitions $\lambda \vdash n$ with the smallest part $ \ge r$ we denote by $p(n,r)$.

Now we will define some special families of partitions. The set ${\cal D}_n$ consists of partitions $\lambda \vdash n$ with distant parts. We say that such a partition has 1-distant parts and we denote $p^{(1)}(n)=|{\cal D}_n|$. Similarly we set ${\cal D}_n^d = \{(\lambda_1, \lambda_2,...,\lambda_l) : \lambda_i-\lambda_{i+1} \ge d\}$, $p^{(d)}(n)=|{\cal D}_d|$. It is said that elements of ${\cal D}_n^2$ have 2-distant parts while elements of ${\cal D}_n^3$ have 3-distant parts.   

Let $l_{i,q}(\lambda)$ denote the number of parts of a partition $\lambda$ congruent to $i\pmod q$ and specially  $l_o(\lambda)=l_{1,2}(\lambda)$ is the number of odd parts of a partition. Let $\hat{p}^{(1)}(n)=|\{ \lambda \in {\cal D}_n :e(\lambda)>2l_o(\lambda)  \}|$, where $e(\lambda)$ is the smallest even part of a partition $\lambda \vdash n$.

It is known \cite{Bre} that the number of partitions of $\lambda \vdash n$ with $d$-distant parts equals the number of partitions $\mu \vdash n$ with $1$-distant parts fulfilling the following constraint: the smallest part that is congruent to $i\pmod d$ is greater than $$d\sum_{j=1}^{i-1} l_{i,q}(\mu),$$ $1 \leq i \leq d$. Namely, given a  partition with $d$-distant parts, subtract $q$ from the second smallest part, than $2q$, $3q$,... from the subsequent smallest parts. During this operation, the parts of resulting partition keeps the same congruence condition with respect to $q$ whereas the weight of the partition is decreased and the $d$-difference is lost. The opposite can be done in a unique way. On the other hand, the similar procedure lead to the same resulting partition starting with a partition with $1$-distant parts satisfying the above condition. This two facts prove the above identity.

Note that for $d$=$2$ the above sum reduces to only one term, resulting with $2l_o(\mu)$. In this case the number of partitions with $2$-distant parts equals the number of partitions with 1-distant parts and with every even part $e(\mu)$ greater than twice the number of odd parts. This polarization, expressed by the next theorem \cite{AE}, is of additional interest since it interprets the left hand side of the first Rogers-Ramanujan identity \cite{BoPa}\cite{MM}.

\begin{theorem}\label{Thm1}
The number of partitions $\lambda \vdash n$ with $2$-distant parts is equal to the number of partitions $\mu \vdash n$ with $1$-distant parts having the smallest even part greater than twice the number of odd parts
\begin{equation}
p^{(2)}(n)=\hat{p}^{(1)}(n).
\end{equation}
\end{theorem}

\begin{proof}
The proof is illustrated by the Young diagrams in Figure \ref{Fig1}. Arrange the diagram in a way the left margine has two dots extra identation per row. Draw the line of justification (of the original diagram) on such a way that in the last row one point remains on the left side of the line. Now rearrange the rows on the right side of the normal line: firstly put odd parts in descending order and then even parts in decreasing order too. As a result we have a new partition starting with even parts. Every row has size at least $2k$--1, where $k$ is counted from the bottom of the diagram. This means, if diagram possessed $k$ odd parts then the smallest even part is at least ($2k$-$1$)+$2$+$1$. 

In order to prove that this correspondence is invertible, we start with adjusted left margin as in the previous consideration. Now we have to explain that all rows intersect the normal line. The condition $e(\mu)>2l_o(\mu)$ ensures that the smallest even part will intersect the line. Having in mind that even parts are mutually different by $2$ or more, all other even parts will intersect the justification line.
\end{proof}

\begin{figure}
\caption{The map (12,10,7,4,1) $\mapsto$ (12,10,6,5,1). Resulting partition consists of $1$-distant parts, having 3 even and two odd parts. \label{Fig1}}
\includegraphics[width=350pt]{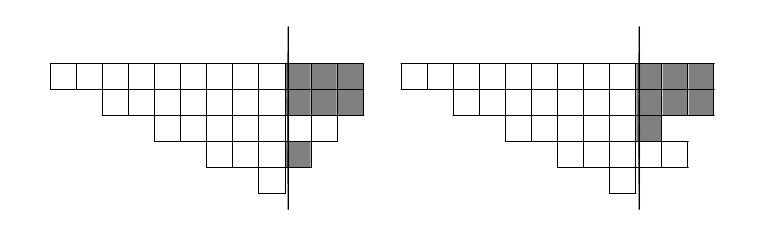}
\end{figure}

For the purpose to take an insight into the nature of partitions $\lambda \vdash n$ with $d$-distant parts, note that they are in one to one correspondence with the partitions $\mu$ such that:
\begin{description}
\item {\it i)} $\mu \in {\cal S}_{n}^d$, ${\cal S}_n^d=\{\mu : \mu \in {\cal P}_n, \mu_{min} \ge 1+(l(\mu)-1)q/2\}$ when $d$ is even, 
\item {\it ii)} $\mu \in {\cal S}_n^d \cap {\cal D}_n $ when $d$ is odd; $q=d-[d=odd]$.
\end{description}
Namely, $1+(1+d)+...+(1+(l-1)d)=l+d{ l \choose 2}=l(1+\frac{(l-1)d}{2})$, for $d$ even; and analogously when $d$ is odd. Let $p^{sd}(n)=|{\cal S}_n^d|$.

In particular, it holds
\begin{eqnarray}
p^{(2)}(n)&=&p^{s2}(n)\\
p^{(3)}(n)&=&p^{(1){s2}}(n)
\end{eqnarray}
meaning that partitions with 2-distant parts are equinumerous to the parititions having exactly one Durfee square. On the other hand, partitions with 3-distant parts are equinumerous to the parititions having exacty one Durfee square but with distant parts. These facts allow to represent $p^{(d)}(n)$ as the sum either of regular partitions (when $d$ is even) or partitions with 1-distant parts ($d$ odd). For $d=2$ and $3$ it follows immediately:
\begin{eqnarray}
p^{(2)}(n)&= &\sum_{i^2 \leq n} \sum_{j \leq i}   p_j(n-i^2)\\
p^{(3)}(n)&= &\sum_{i^2 <n} \sum_{j = i-1,i }   p^{(1)}_j(n-i^2).
\end{eqnarray}

\section{The maximal length of a $d$-distant partition}

Denote the length $l(\lambda)$ of the longest partition $\lambda \in {\cal D}_n^d$ by $l(n,d)$. Obviously, for the partitions having $2$-distant parts the minimal number $n=|\lambda|, \enspace \lambda \in {\cal D}_n^2 $, $l(n,2)$=$2$ is $4$ since $1+3=4$. The minimal $n$ with length $3$ is $9$ since $1+3+5=9$ etc. Thus, for every quadratic number $m$, the maximal length increases by $1$ in respect to the previous quadratic number, $l(m,2)=l((\sqrt{m}-1)^2,2)+1$. Consequently, it holds
\begin{eqnarray}
l(n,2)= \big\lfloor\sqrt{n} \big\rfloor.
\end{eqnarray}
This rule is generalized in the next lemma.

\begin{lemma}
Let $\lambda \in {\cal D}_n^d$. Then for the maximal lenght $l(n,d)$ of $\lambda \vdash n$ it holds
\begin{eqnarray}
l(n,d)= \Bigg\lfloor \frac{d-2+\sqrt{(d-2)^2+8dn}}{2d} \Bigg \rfloor.
\end{eqnarray}
\end{lemma}

\begin{proof}
Let $n$ be the smallest number with property $l(n,d)=l(n-1,d)+1$. Then, the value of the $d$-distant partition $\lambda$, $n \leq |\lambda| < n+q$, $q$ being non-negative integer, corresponds to the largest integer smaller than or equal to the sum of $m$ numbers $1, (1+d), (1+2d),...,(dm-(d-1))$. This sum equals $\frac{m(dm-d+2)}{2}$, which leads to the quadratic equation $dm^2-(d-2)m-2n=0$. Since $l(n,d)=\lfloor m \rfloor$ the statement of lemma follows immediately.
\end{proof}

\section{The main result}

The mentioned Bressoud's result, of which Theorem \ref{Thm1} is a special case, gives identities between partitions with $d$-distant parts and partitions with 1-distant parts having parts separated by certain congruence condition. In particular, when $d$=$2$ this condition is reduced to $\equiv$ $0,1\pmod 2$, separating a partition into even and odd parts. Here we extend these ideas of the polarization of partition. More precisely, we show that there is a one to one correspondence between partitons with $d$-distant parts, $d>2$ and partitions with 1-distant parts having even parts greater than twice the number of odd parts.

\begin{theorem}\label{Thm2}
The number of partitions $\lambda \vdash n$ with $3$-distant parts is equal to the sum of numbers of partition $\mu_i \vdash n-{ i \choose 2}$, $l(\mu_i)=i, \enspace i=1,...,l(n,3)$ with $1$-distant parts having the minimal even part greater than twice the number of odd parts
\begin{equation}
p^{(3)}(n)= \sum_{i \ge 1 } \hat{p}^{(1)}_i(n- { i \choose 2}). 
\end{equation}
\end{theorem}

\begin{proof}
We arange the starting partition $\lambda \vdash n$ in the same manner as in the proof of the Theorem \ref{Thm1}, obtaining on the left side of the line of justification rows shifted by two. Since the partition $\lambda$ consists of 3-distant parts this means that rows on the right side of vertical line are spaced at least by 1, as Figure \ref{Fig2} presents. In case the last row possesses two or more spaces, we have Sylvester triangle with side $l(\lambda)$ on the right side (see Corollary \ref{Cor1} and Figure \ref{Fig3}). However, this is not the general case but the triangle with side $l(\lambda)$--1.

Extracting the triangular partition $\nu_i \vdash \frac{i(i-1)}{2}, \enspace i=l(\lambda)$  form $\lambda$, we obtain partition $\mu_i \in {\cal D}_{|\lambda|-{ i \choose 2}}$. Furthermore, the proof is completed by the same reasoning as in the Theorem \ref{Thm1}.
\end{proof}

\begin{figure}
\caption{The partition $\lambda = (13,9,5,2)$ after extracting the Sylvester's triangle and rearranging rows is mapped to $\mu = (10,6,4,3), \enspace |\mu|=29-6=23$. \label{Fig2}}
\includegraphics[width=\textwidth]{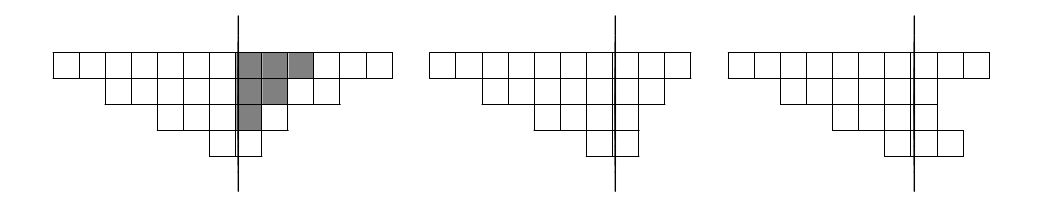}
\end{figure}

Thus, the number of partitions $\lambda \vdash n$ with 2-distant parts equals the sum of partitions $\mu_i \vdash m_i$ with 1-distant parts and with the smallest even part greater than $2l_o(\mu_i)$; $m_i=|\mu_i|=n-i(i-1)/2$, $\l(\mu_i)=i$ for every $i=1,2,...,l(n,3)$ i.e., terms in the sum represent the partition $\mu_i$ that arises by subtracting the Sylvester's triangle $\frac{i(i-1)}{2}$ from the $\lambda$.

As an example, we are going to calculate the number of partitions with 2-distant parts for $n$=$15$. Firstly note that the partition $\lambda =(n)$ will be mapped into itself, which means that the first term in the sum in Theorem \ref{Thm2} is always 1. So, according to the Theorem \ref{Thm2}, for this particular case it holds
\begin{eqnarray*}
p^{(3)}(15)&=&1+\hat{p}^{(1)}_2(14)+\hat{p}^{(1)}_3(12)\\
&=&1+6+3 \\
&=&10.
\end{eqnarray*}
The all 6 partitions of 14 of lenght 2 has even parts greater than the number of odd parts, whereas 3 out of 7 partitions of 12 having lenght 3 are in line with the condition.

\begin{corollary}\label{Cor1}
The number of partitions $\lambda \vdash n$ with $3$-distant parts and with the smallest part at least $2$ is equal to the sum of numbers of partition $\mu_i \vdash n-\frac{i(i+1)}{2}$, $l(\mu_i)=i, \enspace i=1,2,...$ with $1$-distant parts having the minimal even part greater than twice the number of odd parts
\begin{equation}
p^{(3)}(n, 2)= \sum_{i \ge 1 } \hat{p}^{(1)}_i(n- \frac{ i(i+1)}{2}). 
\end{equation}
\end{corollary}

\begin{figure}
\caption{When the smallest part of a partition with 3-distant parts is equal to or greater than $2$ the side of a Sylvester triangle equals the length of the starting partition. \label{Fig3}}
\begin{center}
\includegraphics[width=200pt]{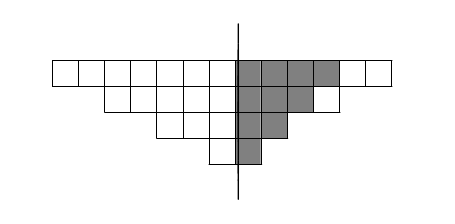}
\end{center}
\end{figure}

Analogy with the Theorem \ref{Thm2} gives identities for partitions with $d$-distant parts. For $d$=$4$ we have 
\begin{equation}
p^{(4)}(n)= \sum_{i \ge 1 } \hat{p}^{(1)}_i(n + i -i^2) 
\end{equation}
with the corollary
\begin{equation}
p^{(4)}(n,2)= \sum_{i \ge 1 } \hat{p}^{(1)}_i(n-  i^2) 
\end{equation}
As the Figure \ref{Fig4} illustrates, for the extracted partition it holds
\begin{eqnarray*}
1+2+3+...+(l-1)&=&{l \choose 2}\\
(d-2)+(2d-4)+(3d-6)+...+(l-1)(d-2)&=&(d-2){l \choose 2}\\
\end{eqnarray*}
for $d$=$3$ and for general case, respectively, where $l$ is the length of a starting partition. So, for any $d$ we keep $2$-distant parts on the left side of the justification line, which means that these parts are odd. After applying {\it i)}  subtracting partition $\nu_i \vdash (d-2){ i \choose 2}, \enspace i=1,...,l(n,d)$ and {\it ii)}  sorting parts on the right side of the justification line, we get a partition polarized into even and odd parts. This proves the next theorem.

\begin{figure}
\caption{Extracted partitions for $d=3,4,5$, respectively. \label{Fig4}}
\includegraphics[width=380pt]{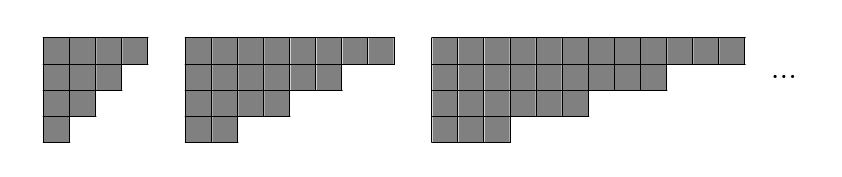}
\end{figure}

\begin{theorem}\label{Thm3}
Partitions $\lambda \vdash n$ with $d$-distant parts are equinumerous to partitions $\mu_i \vdash n-(d-2){ i \choose 2}$, $l(\mu_i)=i, \enspace i=1,...,l(n,d)$ with $1$-distant parts having minimal even part greater than twice the number of odd parts
\begin{equation}
p^{(d)}(n)= \sum_{i \ge 1 } \hat{p}^{(1)}_i(n- (d-2){ i \choose 2}). 
\end{equation}
\end{theorem}

Imagine a partition with $3$-distant parts and a partition with $3$-distant parts where the minimal part is at least $2$. According to the Theorem \ref{Thm2} and the Corollary \ref{Cor1} there is one to one correspondence between these partitions and a certain polarized partitions. Note that if the starting partitions are of the same length $l$ then the difference between extracted parts is exactly $l$. Obviously, the difference $l$ between extracted parts is always the case when starting partitions have mentioned characteristics. Thus, the generalization of the previous corollaries is as follows:
\begin{equation}
p^{(d)}(n, 2)= \sum_{i \ge 1} \hat{p}^{(1)}_i(n-i-  (d-2){ i \choose 2}). 
\end{equation}
A similar identity holds for the general case i.e., when there is any constraint on the minimal part of a partition. Let the starting partition $\lambda$, $l(\lambda)=l$ in our bijection has minimal part equal to $r$. Then the parts on the right side of the justification line that we subtract from the starting partition are $r-1, r-1+(d-2), r-1+2(d-2),...,r-1+(l-1)(d-2)$. In other words, we subtract $r-1$ rows of the hight $l$ more than in the case we consider in the previous theorem i.e., the subtracted part is $(r-1)l+  (d-2){ l \choose 2}$. This reasoning leads to the next corollary.

\begin{corollary}\label{Cor2}
Partitions $\lambda \vdash n$ with $d$-distant parts and with the smallest part at least $r$ are equinumerous to partitions $\mu_i \vdash n-(r-1)i-\frac{i(i+1)}{2}$, $\enspace i=1,2,...$ with $i$ $1$-distant parts having the minimal even part greater than twice the number of odd parts
\begin{equation}
p^{(d)}(n, r)= \sum_{i \ge 1} \hat{p}^{(1)}_i(n-(r-1)i-  (d-2){ i \choose 2}). 
\end{equation}
\end{corollary}

Results demonstrated in this paper show that the number $p^{(d)}(n, r)$, $d \ge2, r \ge1, n\ge 0$ can be represented as the sum of numbers of polarized partitions of a certain natural number, with $k$ $1$-distant parts. In other words, there is a one to one correspondence between polarized partitions and partitions with $d$-distant parts and possibly with a constraint on the minimal part.


\begin{thebibliography}{6}

\bibitem{Aig} Aigner M., A Course in Enumeration, Springer-Verlag, Berlin Heidelberg, 2007

\bibitem{AE} Andrews G.E., Eriksson K., Integer Partitions, Cambridge University Press, Cambridge, 2004

\bibitem{BoPa} Boulet C., Pak I.,  Combinatorial proof of the Rogers-Ramanujan and Schur indentities, J. Combin. Th., 2006, 113, 1019--1030

\bibitem{Bre} Bressoud D.M., A new family of partition identities, Pacific J. Math., 1978, 77, 71--74

\bibitem{MM} MacMahon A., Combinatory Analysis Vol 2, Cambridge University Press, Cambridge, 1916  (Reprinted: Chelsea, New York, 1960) 

\bibitem{Pak} Pak I., Partition bijections, a survey, Raman. J., 2006, 12, 5--75



\end{thebibliography}
\end{document}